\documentclass{custom} % Anonymized submission
\title[Query complexity for strongly log-concave in one dimension]{The query complexity of sampling from strongly log-concave distributions in one dimension}

\usepackage{enumitem}
\usepackage{graphicx}
\usepackage{mathtools}
\usepackage{microtype}
\usepackage{times}
\usepackage[Symbolsmallscale]{upgreek}
\usepackage{thm-restate} 

\usepackage{hyperref}       % hyperlinks
\usepackage{url}            % simple URL typesetting
\usepackage{booktabs}       % professional-quality tables
\usepackage{amsfonts}       % blackboard math symbols
\usepackage{nicefrac}       % compact symbols for 1/2, etc.
\usepackage{microtype}      % microtypography

\usepackage{algorithm}
\usepackage[noend]{algpseudocode}
\usepackage{amsmath, amssymb}
\usepackage{bbm}
\usepackage{mathrsfs}
\usepackage{mathtools}
\usepackage{ragged2e}
\usepackage{thmtools}
\usepackage{thm-restate}
\usepackage{tikz}
\usepackage{wrapfig}
\usepackage{style}

\allowdisplaybreaks{}
\newcommand{\ttt}{\kappa^{-\frac{1}{2}}}
\DeclareMathOperator{\unif}{unif}

\makeatletter
\def\set@curr@file#1{\def\@curr@file{#1}} %temp workaround for 2019 latex release
\makeatother
\usepackage[load-configurations=version-1]{siunitx} % newer version

\coltauthor{\Name{Sinho Chewi} \Email{schewi@mit.edu}\\
\Name{Patrik Gerber} \Email{prgerber@mit.edu}\\
 \Name{Chen Lu} \Email{chenl819@mit.edu}\\
 \Name{Thibaut {Le Gouic}} \Email{tlegouic@mit.edu}\\
 \Name{Philippe Rigollet} \Email{rigollet@mit.edu}\\
 \addr MIT}

\begin{document}

\maketitle

\begin{abstract}
    We establish the first tight lower bound of $\Omega(\log\log\kappa)$ on the query complexity of sampling from the class of strongly log-concave and log-smooth distributions  with condition number $\kappa$ in one dimension. Whereas existing guarantees for MCMC-based algorithms scale polynomially in $\kappa$, we introduce a novel algorithm based on rejection sampling that closes this doubly exponential gap.
%     Motivated by the oracle lower bounds initiated by Nemirovsky and Yudin for optimization, w
    
%     This paper establishes the first lower bound for the complexity of sampling from a structured distribution in terms of the number of queries to a natural first order oracle commonly employed in such tasks. More precisely, in one dimension, we show that sampling from a strongly log-concave and log-smooth distribution with condition number $\kappa$ requires $\Omega(\log\log\kappa)$ calls to any local oracle. Moreover, this lower bound us sharp and can be achieved using a novel rejection sampling algorithm which improves on state-of-the-art subroutines such as the ones appear in the Hit-and-Run algorithms. 
    
%   This paper establishes the first lower complexity bound for the task of sampling from a strongly log-concave and log-smooth distribution in one dimension. This lower bound of $\Omega(\log\log\kappa)$, where $\kappa$ is the condition number, is achieved via a novel rejection sampling algorithm. We also give an application of our algorithm to the implementation of the Hit-and-Run sampling method.
\end{abstract}

\section{Introduction}

The task of sampling from a target probability distribution known up to a normalizing constant is of fundamental importance in fields such as Bayesian statistics, randomized algorithms, and online learning. Recently, there has been a resurgence of interest in sampling and its interplay with the more well-developed field of optimization. On the one hand, the extensive optimization toolkit has inspired the development of novel sampling algorithms~\citep{bernton2018jko, wibisono2018samplingoptimization, wibisono2019proximal, chewietal2020mirrorlangevin, dingetal2020coordinatelangevin, salimkorbaluise2020wassproximal, maetal2021nesterovmcmc}; on the other hand, the theory of optimization has motivated researchers to provide quantitative and non-asymptotic convergence guarantees for sampling methods, which depend on parameters that describe the problem complexity such as the condition number and the dimension~\citep{durmus2017nonasymptotic, dalalyankaragulyan2019userfriendly, chewietal2020mala}.

Conspicuously absent from this interplay, however, are lower complexity bounds for sampling, in analogy to the oracle lower bounds initiated in the seminal work by Nemirovsky and Yudin for optimization~\citep{nemirovskyyudin}. Besides charting the fundamental limits of optimization, such lower bounds have been instrumental in the development of faster algorithms, most notably Nesterov's acceleration, which was ``found mainly because the investigating of complexity enforced to believe that such a method should exist''~\citet[Ch.\ 10]{Nem94}.
% allow us to identify algorithms which are \emph{minimax optimal}, in the sense that their performance cannot be improved uniformly over a given problem class. Moreover, lower bounds illuminate the fundamental difficulty of the problem.

A canonical structured class of distributions is that of strongly log-concave and log-smooth distributions on $\R^d$, i.e., the class of distributions with a density $p \propto \exp(-V)$, where the potential $V : \R^d\to\R$ is twice continuously differentiable, $\alpha$-strongly convex, and $\beta$-smooth.
The relevant parameters of this class are the dimension $d$, as well as the \emph{condition number} $\kappa := \beta/\alpha$, and we seek to understand the number of queries to $V$ (and its derivatives) necessary to generate a sample close in total variation distance to $p$. We call a solution to this problem a \emph{general sampling lower bound}.

\paragraph{Related works.} Despite several attempts at establishing query complexity lower bounds for sampling, we are not aware of a general sampling lower bound. Whereas sampling upper bounds are derived using techniques that are close to those employed in optimization~\citep{dalalyan2014theoretical, durmusmajewski2019lmcconvex}, it is unclear how to use lower bound techniques for optimization~\citep{nesterov2013introductory} to derive general sampling lower bounds.
Note that sampling upper bounds typically assume that the minimizer of $V$ is known a priori; thus, a direct reduction of the sampling task to apply existing optimization lower bounds would likely capture the complexity of finding the mode of $V$ rather than the intrinsic difficulty of the sampling task itself. In lieu of a direct reduction, it is possible to envision, at least in principle, an approach which adapts the optimization lower bound constructions to the sampling setting, but we are not aware of any successful results in this direction.

Another family of approaches is based on information-theoretic ideas which have been highly successful for developing a minimax theory of statistics~\citep{lecam1986asymptotic, lecamyang2000asymptotics, tsybakov2009nonparametric}; however, prior works applying these ideas have
largely focused on various adjacent questions which do not imply a lower bound for the sampling task itself. A notable example is the estimation of the normalizing constant of a strongly log-concave distribution, for which a lower bound was established in~\cite{ge2020estimating}.
However, this lower bound does not yield a general sampling lower bound; in fact, the two problems differ in difficulty. Indeed, the randomized midpoint discretization of the underdamped Langevin dynamics~\citep{shen2019randomized} obtains samples in $\mc O(d^{1/3})$ queries, whereas the lower bound for estimating the normalizing constant in~\cite{ge2020estimating} grows as $\tilde\Omega(d)$.
Another example is the paper~\cite{chatterjibartlettlong2020oraclesampling}, which studies sampling with access to stochastic gradient queries. However, the resulting lower bound arises primarily out of the need to overcome the noise in the gradient queries, and it again does not yield sampling lower bounds for our setting of precise gradient queries.

To circumvent the difficulties in establishing general sampling lower bounds, various works have focused on establishing lower bounds for specific and popular algorithms such as the Underdamped Langevin Algorithm (ULA)~\citep{cao2020complexity} and the Metropolis-Adjusted Langevin Algorithm (MALA)~\citep{chewietal2020mala}. In particular, the last paper establishes that the exact dimension dependence of the query complexity for MALA over the class of strongly log-concave and log-smooth distributions is $\Theta(\sqrt{d})$, although the optimal dependence on $\kappa$ is unknown.

\paragraph{A lower complexity bound in one dimension.}
Recall that for convex optimization, there are two relevant regimes~\cite[see, e.g.,][]{bubeck2015convex}: (1) the low-dimension regime, in which algorithms such as the cutting plane method achieve the rate $\mc O(d\log(1/\varepsilon))$ (where $\varepsilon$ is the accuracy parameter), and (2) the high-dimensional regime, in which algorithms such as gradient descent achieve dimension-free rates at the cost of inverse polynomial dependency on the accuracy. In this paper, we study the low-dimensional regime for sampling; in particular, we consider $d=1$.

We prove that for the class of $\alpha$-strongly log-concave and $\beta$-log-smooth distributions in one dimension (with mode at $0$), any algorithm, which can produce a sample that is at total variation distance at most $\frac{1}{64}$ from the target distribution $p$ (uniformly over $p$ belonging to the class), must make at least $\Omega(\log\log\kappa)$ queries to $V$ or any of its derivatives. To our knowledge, this is the first lower complexity bound for this problem class.

\paragraph{Achievability of the lower bound.} The lower bound of $\Omega(\log \log \kappa)$ is surprisingly small, and existing guarantees for standard algorithms such as the Langevin algorithm (or its variants), the Metropolis-Adjusted Langevin Algorithm, or Hamiltonian Monte Carlo, all have a dependence that scales polynomially with the condition number $\kappa$ \cite[see for instance the comparison in][]{shen2019randomized}.

To provide an algorithm which matches the lower bound, we return to the fundamental idea of rejection sampling, developed by Stan Ulam and John von Neumann \citep{vonNeumann1951, Ulam}. We develop an algorithm which uses $\mc O(\log\log\kappa)$ queries in order to build a proposal distribution. Once the proposal distribution is constructed, new samples which are $\varepsilon$-close to $p$ in total variation distance can be generated using $\mc O(\log(1/\varepsilon))$ additional queries per sample.

Although our algorithm is tailored to distributions in one dimension, the task of sampling from a one-dimensional log-concave distribution is an important subroutine for higher dimensional algorithms, such as the Hit-and-Run algorithm. We describe the application of our algorithm to Hit-and-Run in Section~\ref{scn:hit_and_run}.

\section{Lower bound}

We begin by formally defining the class of strongly log-concave and log-smooth distributions in one dimension, which is the focus of this paper.

\begin{definition}\label{defn:strongly_lc}
The class of univariate \emph{$\alpha$-strongly log-concave and $\beta$-log-smooth} distributions, for constants $0 < \alpha \le \beta$, is the class of continuous distributions $p$ supported on $\R$, whose density is of the form $p(x) = \exp(-V(x))$, for a potential  function $V : \R \to\R \cup\{\infty\}$ which is twice continuously differentiable and satisfies
\begin{align}\label{eq:smooth_strong_cvx}
    \alpha \le V''(x) \le \beta \, , \qquad \forall x\in \R\, .
\end{align}
In addition, we always assume\footnote{This localization assumption is common in the sampling literature; without some knowledge of the mode (e.g.\ that the mode is contained in an interval) it is impossible to even find the mode in the query model.} that the mode of the distribution is at $0$, or equivalently $V'(0) = 0$.
\end{definition}

We study the \emph{query complexity} of sampling from this class.
Formally, suppose that the target distribution is $p = \exp(-V)$. The sampling algorithm is allowed to make queries to the following oracle: given a point $x\in\R$, the oracle returns some or all of (1) the evaluation of the potential $V(x) + C$ up to a constant $C$, which is unknown to the algorithm but does not change from query to query; (2) the evaluation of the gradient $V'(x)$; or (3) the evaluation of the Hessian $V''(x)$. Depending on what information the oracle returns, it may be described as providing $0^{\rm th}$-, $1^{\rm st}$-, or $2^{\rm nd}$-order information. For instance, the Langevin algorithm uses $1^{\rm st}$-order information, whereas the Metropolis-Adjusted Langevin Algorithm uses both $0^{\rm th}$-order and $1^{\rm st}$-order information. Our lower bound will in fact apply to the strongest of these oracles, namely the one that returns all three pieces of information.

We now state our lower bound.

\begin{theorem}\label{thm:general_lower}
    Consider the class $\mc P$ of univariate $\alpha$-strongly log-concave and $\beta$-log-smooth distributions as defined in~\autoref{defn:strongly_lc}, and let $\kappa := \beta / \alpha$ denote the condition number. Suppose that an algorithm satisfies the following guarantee: for any $p\in \mc P$, the algorithm makes $n$ queries to the oracle providing $0^{\rm th}$-, $1^{\rm st}$-, and $2^{\rm nd}$-order information for $p$, and outputs a random variable whose law is at most $\frac{1}{64}$ away from $p$ in total variation distance.
    Then, $n \gtrsim \log\log \kappa$.
\end{theorem}

We now give some intuition for the lower bound construction, and defer the proof to Appendix~\ref{scn:stronglylc_lower}. The strategy is to construct a family of distributions $\{p_1. \ldots, p_m\}$  which forms a packing of the class $\mc P$ in total variation distance. Because the family is well-separated, if an algorithm can accurately sample from each $p_i$, it can also \emph{identify}  $p_i$. We construct the family ${\{p_i\}}_{i=1}^m$ in such a way that identifying $p_i$ from queries to low-order oracles requires at least $\Omega(\log m)$ queries, e.g., via bisection.

With the strategy in place, we now describe motivation for the construction of the family ${\{p_i\}}_{i=1}^m$. Suppose that we have a distribution $p\propto \exp(-V)$ which is rescaled to satisfy $1 \le V'' \le \kappa$. The bound $V'' \ge 1$ implies that a substantial fraction of the mass of $p$ is supported on the interval $[-1, 1]$. On the other hand, the bound $V'' \le \kappa$ allows for the density $p$ to suddenly drop from $\approx 1$ to nearly $0$ over an interval of much smaller length, $\asymp 1/\sqrt\kappa$. Hence, as a first approximation, we can imagine dividing the interval $[-1,1]$ into $\asymp \sqrt\kappa$ bins, and thinking of each $p_i$ as piecewise constant on each bin. While keeping the log-concavity constraint in mind, for the purpose of this heuristic discussion we will consider the family ${\{p_i\}}_{i=1}^m$ of $m \asymp \sqrt\kappa$ distributions, where $p_i$ is the uniform distribution on $[-i/\sqrt\kappa, i/\sqrt\kappa]$; see Figure~\ref{fig:unif}.

\begin{figure}
    \centering
    \begin{tikzpicture}
  \draw[<->] (-6, 0) -- (6, 0) node[right] {};
  \draw[->] (0, 0) -- (0, 1.5) node[above] {};
  \draw[blue] (-1, 0) -- (-1, 1) -- (1, 1) -- (1, 0);
  \draw[blue] (-2, 0) -- (-2, 0.5) -- (2, 0.5) -- (2, 0);
  \draw[blue] (-3, 0) -- (-3, 0.333) -- (3, 0.333) -- (3, 0);
  \draw[blue] (-4, 0) -- (-4, 0.25) -- (4, 0.25) -- (4, 0);
  \draw[blue] (-5, 0) -- (-5, 0.2) -- (5, 0.2) -- (5, 0);
\end{tikzpicture}
    \caption{A family of uniform distributions.}
    \label{fig:unif}
\end{figure}
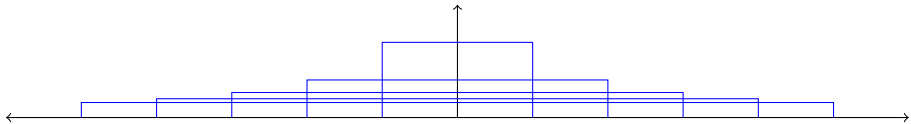

However, this family is not well-separated in total variation distance. Indeed, it can be checked that for $i < j$, in order for the total variation distance between $p_i$ and $p_j$ to be appreciable, we require $j\ge 2i$. This motivates us to consider the subfamily $\{p_{2^i}\,, 1\le i \le \log_2 \sqrt{\kappa}\}$, of which there are $\mc O(\log\kappa)$ elements.
For this subfamily, we can hope to reduce the task of sampling to that of identifying $p_{2^i}$ via queries, and binary search for this problem requires only $\mc O(\log\log\kappa)$ queries. This is the basis for our somewhat unusual lower bound.

The uniform distributions involved in this informal discussion do not belong to the class $\mc P$, as they are neither strongly log-concave nor log-smooth. The main technical challenge in our lower bound is to produce distributions which lie in $\mc P$ but still behaves similarly to uniform distributions, in the sense of requiring $\Omega(\log\log\kappa)$ oracle queries to identify a distribution via queries. We defer these details to the appendix.

\section{Upper bound}

In this section, we show that the $\Omega(\log\log\kappa)$ lower bound in the previous section is achievable. Note that the existing guarantees for standard sampling algorithms (c.f.\ the comparison in~\cite{shen2019randomized}) usually scale polynomially in the condition number $\kappa$, so they are not optimal for our setting.

Moreover, the heuristic discussion of the lower bound construction motivates choosing the query points according to a binary search strategy. In order to implement this idea, we turn towards the classical idea of rejection sampling: first, we make queries in order to construct a \emph{proposal} distribution~$q$. To generate new samples from $p$, we repeatedly draw samples from $q$, and each sample is accepted with a carefully chosen acceptance probability (which can be computed via additional queries to the oracle for the density up to normalization).

\begin{algorithm}[H]
  \caption{{\sc Envelope}}\label{ALG:upper}
      \begin{algorithmic}[1]
    % \Procedure{Envelope}{}
    \State{use binary search to find the first index $i_+ \in \{0,1,\dotsc,\lceil \frac{1}{2} \log_2 \kappa \rceil\}$ for which $V(2^{i_+}/\sqrt\kappa) \ge \frac{1}{2}$}
    \State{use binary search to find the first index $i_- \in \{0,1,\dotsc,\lceil \frac{1}{2} \log_2 \kappa \rceil\}$ for which $V(-2^{i_-}/\sqrt\kappa) \ge \frac{1}{2}$}
    \State set $x_- := -2^{i_-}/\sqrt\kappa$ and $x_+ := 2^{i_+}/\sqrt\kappa$
    \State \Return {\begin{align*}%\label{eq:upper_envelope}
    \tilde q(x)
    &:= \begin{cases} \displaystyle\exp\bigl[-\frac{x-x_-}{2x_-} - \frac{(x-x_-)^2}{2}\bigr]\,, & x \le x_-\,, \\ 1\,, & x_- \le x \le x_+\,, \\ \displaystyle\exp\bigl[-\frac{x-x_+}{2x_+} - \frac{(x-x_+)^2}{2}\bigr]\,, & x \ge x_+\,.  \end{cases}
\end{align*}}
    %\State\Return $(-2^{i_-}/\sqrt\kappa, 2^{i_+}/\sqrt\kappa)$
    % \EndProcedure
  \end{algorithmic}
\end{algorithm}
% \end{minipage}

We give the high-level pseudocode for building an upper envelope in Algorithm~\ref{ALG:upper}, and for generating new samples in Algorithm~\ref{ALG:sample}. Note that while our lower bound applies to algorithms using $0^{\rm th}$-, $1^{\rm st}$-, and $2^{\rm nd}$-order information, our upper bound algorithm in fact only requires $0^{\rm th}$-order information. We next proceed to discuss details of the algorithms.

Before implementing Algorithm~\ref{ALG:upper}, we first perform several preprocessing steps. Recall that the mode of the distribution $p$ is assumed to be at $0$, and that $p \propto \exp(-V)$. We also assume that $1 \le V'' \le \kappa$. To reduce to this case, say we start with $\alpha \le V'' \le \beta$, and the bounds $\alpha$, $\beta$ are known. Then, observe that the rescaled potential $\bar V(x) := V(x/\sqrt\alpha)$ satisfies $1 \le \bar V'' \le \kappa = \beta/\alpha$. Given access to an oracle for $V$ (up to additive constant), we can simulate an oracle to $\bar V$ (up to additive constant) and apply our algorithm to generate a sample $\bar X$ from the density $\bar p\propto \exp(-\bar V)$; it can be checked that $\bar X/\sqrt \alpha$ is a sample from $p$. Finally, we assume that the oracle, when given a query point $x$, returns $V(x)$, where $V$ is normalized to satisfy $V(0) = 0$; this is achieved by replacing the output $V(x)$ of the oracle by $V(x) - V(0)$.

\begin{wrapfigure}{R}{0.4\textwidth}
\vspace{-.8cm}
\begin{minipage}{0.4\textwidth}
\begin{algorithm}[H]
  \caption{{\sc Sample}}\label{ALG:sample}
  \begin{algorithmic}[1]
    % \Procedure{Sample}{}
    \State normalize $\tilde q$ to form $q$
    \While{sample is not accepted}
    \State sample $X \sim q$
    \State accept $X$ w.p.\ $\tilde p(X)/\tilde q(X)$
    \EndWhile
    \State\Return{$X$}
    % \EndProcedure
  \end{algorithmic}
\end{algorithm}
\end{minipage}
\end{wrapfigure}

Implementing the first step of Algorithm~\ref{ALG:upper} requires performing binary search over an array of size $\mc O(\log \kappa)$, which requires only $\mc O(\log\log\kappa)$ queries; similar comments apply to the second step. We prove in Appendix~\ref{scn:stronglylc_upper} that the indices $i_-$ and $i_+$ always exist under our assumptions.
% The output of Algorithm~\ref{ALG:upper} is the functio
% \begin{align}\label{eq:upper_envelope}
%     \tilde q(x)
%     &:= \begin{cases} \displaystyle\exp\bigl[-\frac{x-x_-}{2x_-} - \frac{(x-x_-)^2}{2}\bigr]\,, & x \le x_-\,, \\ 1\,, & x_- \le x \le x_+\,, \\ \displaystyle\exp\bigl[-\frac{x-x_+}{2x_+} - \frac{(x-x_+)^2}{2}\bigr]\,, & x \ge x_+\,.  \end{cases}
% \end{align}
We prove in Appendix~\ref{scn:stronglylc_upper} that the output $\tilde q$ of Algorithm~\ref{ALG:upper} is an \emph{upper envelope} for the oracle, i.e., $\tilde q \ge \exp(-V)$. The upper envelope $\tilde q$ constructed in Algorithm~\ref{ALG:upper} is the input to Algorithm~\ref{ALG:sample}; see Figure~\ref{fig:upper_env}.

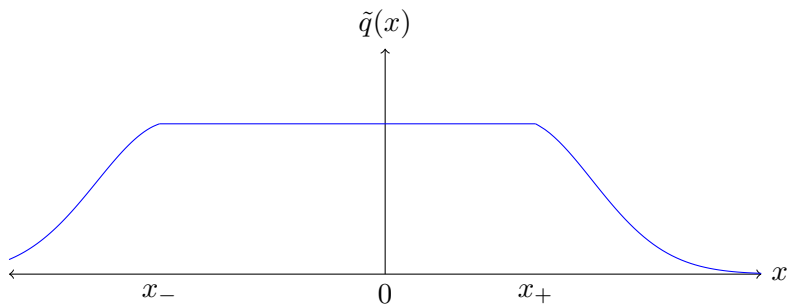
\begin{figure}
    \centering
    \begin{tikzpicture}
  \draw[<->] (-5, 0) -- (5, 0) node[right] {$x$};
  \draw[->] (0, 0) -- (0, 3) node[above] {$\tilde q(x)$};
  \draw[blue] (-3, 2) -- (2, 2);
  \draw[domain=-5:-3, smooth, variable=\x, blue] plot ({\x}, {2*exp(-(\x+3)/(-6) - (\x + 3) * (\x + 3)/2)});
  \draw[domain=2:5, smooth, variable=\x, blue] plot ({\x}, {2*exp(-(\x-2)/4 - (\x -2) * (\x -2)/2)});
  \node at (0, 0) [below] {$0$};
  \node at (-3, 0) [below] {$x_-$};
  \node at (2, 0) [below] {$x_+$};
\end{tikzpicture}
    \caption{The upper envelope $\tilde q$ constructed in Algorithm~\ref{ALG:upper}.} %\patrik{The right gaussian tail seems to cross the x-axis}} \sinho{I don't think it crosses, it just gets very close}
    \label{fig:upper_env}
\end{figure}

In Algorithm~\ref{ALG:sample}, we normalize $\tilde q$ to a probability distribution $q$, which requires computing a one-dimensional integral for the normalizing constant: $\int_\R \tilde q$. Once normalized, we must also be able to draw samples from the distribution $q$. These steps can be implemented with low computational burden, but we do not dwell on this point here because we are primarily interested in the \emph{query complexity} in this work. Note that the steps of normalizing $q$ and drawing new samples from $q$ do not require additional queries to the oracle.

The framework of rejection sampling provides a flexible guarantee: if we desire an \emph{exact} sample from $p$, then we can continue drawing samples from $q$ until one is accepted, yielding an exact sample with a guarantee on the expected total number of queries. On the other hand, if we are content with producing a sample whose law is at a fixed distance $\varepsilon$ away from $p$ in total variation distance, then we can force the algorithm to stop after a prespecified number of iterations, declaring failure if no sample from $q$ is accepted, and achieve the total variation guarantee. We describe both of these guarantees in the following theorem, which summarizes the query complexity of our algorithm.

\begin{theorem}\label{thm:upper_bd}
    Suppose that the target distribution $p$ belongs to the class of univariate strongly log-concave and log-smooth distributions (\autoref{defn:strongly_lc}).
    Algorithm~\ref{ALG:upper} uses $\mc O(\log\log\kappa)$ queries to build the upper envelope $\tilde q$. Once $\tilde q$ is constructed, we can use it for either of the following tasks.
    \begin{enumerate}
        \item (exact sampling) Algorithm~\ref{ALG:sample} returns an exact sample from $p$ after an additional $\mc O(1)$ expected queries to the oracle.
        \item (approximate sampling) Fix an accuracy parameter $0 < \varepsilon < 1$. If we limit Algorithm~\ref{ALG:sample} to use at most $\mc O(\log(1/\varepsilon))$ queries, then the output of Algorithm~\ref{ALG:sample} (or `FAILURE', if Algorithm~\ref{ALG:sample} fails to accept a sample within the allowed number of queries) has a distribution which is at total variation distance at most $\varepsilon$ away from $p$.
    \end{enumerate}
\end{theorem}

We give the proof in Appendix~\ref{scn:stronglylc_upper}.

\section{Application to Hit-and-Run}\label{scn:hit_and_run}

Although our upper bound algorithm applies only to univariate distributions, in this section we demonstrate how to use our algorithm as a subroutine for the well-known Hit-and-Run algorithm for sampling from high-dimensional log-concave distributions.

Hit-and-Run algorithms form a class of Markov chain Monte Carlo methods~\citep{belisle1993hit}. Given a target distribution with density $p:\R^d \to \R$ and a distribution $\nu$ on the sphere $\mathbb{S}^{d-1}$, one can construct a Markov chain with stationary distribution $p$ as follows. If the current point of the Markov chain is at the point $x_t \in \R^d$ at iteration $t$, we
\begin{enumerate}
    \item choose a random direction $u \sim \nu$,
    \item\label{hit-and-run step 2} choose a random step size $\lambda$ from the distribution with density $\propto p(x_t + \lambda u)$,
    \item and set $x_{t+1} = x_t + \lambda u$.
\end{enumerate}
It is not hard to see that the above dynamics are reversible with respect to $p$. Lovász and Vempala \citep{lovasz1999hit, lovasz2003hit} study this algorithm when $p$ is log-concave, and establish rapid mixing of the chain when the direction distribution is uniform. However, they only partially addressed the issue of implementation, suggesting the use of an approximate algorithm for Step~\ref{hit-and-run step 2}. In the case when $p$ is strongly log-concave and log-smooth, our methods allow us to resolve this issue and provide an efficient and exact sampling algorithm.

\begin{proposition}\label{prop:hit-and-run}
If $p$ is $1$-strongly log-concave and $\kappa$-log-smooth, then Step \ref{hit-and-run step 2} can be implemented exactly, with amortized query and time complexity $\mc O(\log(\kappa d))$.
\end{proposition}
We defer the description of the algorithm, and the proof of Proposition \ref{prop:hit-and-run} to Appendix \ref{sec:proof of hit and run}. We remark that the query complexity is actually dominated by the cost of finding the maximum of $p$ when restricted to the line, and the sampling step is comparatively cheap.

\section{Conclusion and outlook}

In this paper, we established the oracle complexity of sampling from the class of univariate strongly log-concave and log-smooth distributions, in analogy with the now pervasive oracle lower bounds for optimization initiated by Nemirovsky and Yudin~\citep{nemirovskyyudin}. A clear future direction suggested by this work is to extend this result to higher dimensions, and to ultimately develop a theory of lower complexity bounds and optimal algorithms for sampling.

Recently, an intense amount of research has been devoted to the use of Markov chain Monte Carlo-based methods for sampling, and it may come as a surprise that the complexity lower bound we have proven in this paper is attained by an entirely different type of algorithm, namely rejection sampling. Our result highlights that standard algorithms may not be optimal, and that the search for optimal algorithms goes hand-in-hand with lower bound constructions.

In particular, our work motivates revisiting the idea of rejection sampling through the modern lens of minimax optimality. 

\section*{Acknowledgments}

Sinho Chewi was supported by the Department of Defense (DoD) through the National Defense Science \& Engineering Graduate Fellowship (NDSEG) Program. Thibaut Le Gouic was supported by NSF award IIS-1838071. Philippe Rigollet was supported by NSF awards IIS-1838071, DMS-1712596,  and DMS-2022448.

\bibliography{ref.bib}

\pagebreak

\appendix

\section{Proof of the lower bound}\label{scn:stronglylc_lower}

\subsection{The construction}

%Let $c>1$ be a large constant, and let $\delta < 1$ be a small positive constant, both to be chosen later.
Let $m$ be the largest integer such that 
\begin{align}\label{eq:lower_bd_defn_m}
    \exp\bigl(-\frac{2^{2m-2}}{2 \kappa}\bigr) \ge \frac{1}{2}\,.
\end{align}
Define two auxiliary functions
\begin{align*}
    \phi(x) := \begin{cases} \kappa\,, & 1/2 \le x < 1\,, \\ 1\,, & 1 \le x < 2\,, \\ \kappa\,, & 2 \le x < 5/2\,, \\ 0\, & \text{otherwise}\,, \end{cases} \qquad \psi(x) := \begin{cases} 1\,, & 5/2 \le x < 4\,, \\ \kappa\,, & 4 \le x < 5\,, \\ 0\,, & \text{otherwise}\,. \end{cases}
\end{align*}
We define a family ${(V_i)}_{i=1}^m$ of $1$-strongly convex and $\kappa$-smooth potentials as follows.
We  require that $V_i(0) = V_i'(0) = 0$ and that $V_i$ be an even function, so it suffices to specify $V_i''$ on $\R_+$. 
The second derivative is given by
\begin{align*}
    V_i''(x)
    &:= \one\bigl\{x \le \ttt 2^{i-1}\} + \phi\bigl( \frac{x}{\ttt 2^i} \bigr) + \sum_{j=i}^{m-1} \psi\bigl( \frac{x}{\ttt 2^j} \bigr) + \one\{x \ge 5\ttt 2^{m-1}\}\,, \qquad x \ge 0\,.
\end{align*}
Observe that all of the terms in the above summation have disjoint supports, see Figure~\ref{fig:second_deriv}.

\begin{figure}
    \centering
    \begin{tikzpicture}
  \draw[->] (0, 0) -- (12, 0) node[right] {$x$};
  \draw[->] (0, 0) -- (0, 5) node[above] {$V_i''(x/\sqrt \kappa)$};
  \draw[blue] (0, 1) -- (1, 1);
  \draw[blue, dashed] (1, 4) -- (2, 4);
  \draw[blue, dashed] (2, 1) -- (3.5, 1);
  \draw[blue, dashed] (3.5, 4) -- (5, 4);
  \draw[blue, dotted, thick] (5, 1) -- (6.5, 1);
  \draw[blue, dotted, thick] (6.5, 4) -- (8, 4);
  \draw[blue] (11, 1) -- (12, 1);
  \node at (0, 1) [left] {$1$};
  \node at (0, 4) [left] {$\kappa$};
  \node at (1, 0) [below] {$2^{i-1}$};
  \node at (2, 0) [below] {$2^i$};
  \node at (3.5, 0) [below] {$2^{i+1}$};
  \node at (5, 0) [below] {$\frac{5}{4}2^{i+1}$};
  \node at (6.5, 0) [below] {$2^{i+2}$};
  \node at (8, 0) [below] {$\frac{5}{4}2^{i+3}$};
  \node at (9.5, -0.25) [below] {$\dots$};
  \node at (11, 0) [below] {$\frac{5}{4} 2^{m+2}$};
  % V_{i+1}
%   \draw[orange] (0, 0.97) -- (2, 0.97);
%   \draw[orange, dashed] (2, 4) -- (3.5, 4);
%   \draw[orange, dashed] (3.5, 0.97) -- (6.5, 0.97);
%   \draw[orange, dashed] (6.5, 3.97) -- (8, 3.97);
\end{tikzpicture}
    \caption{The dashed lines correspond to $\phi$ and the dotted lines correspond to $\psi$.}
    \label{fig:second_deriv}
\end{figure}
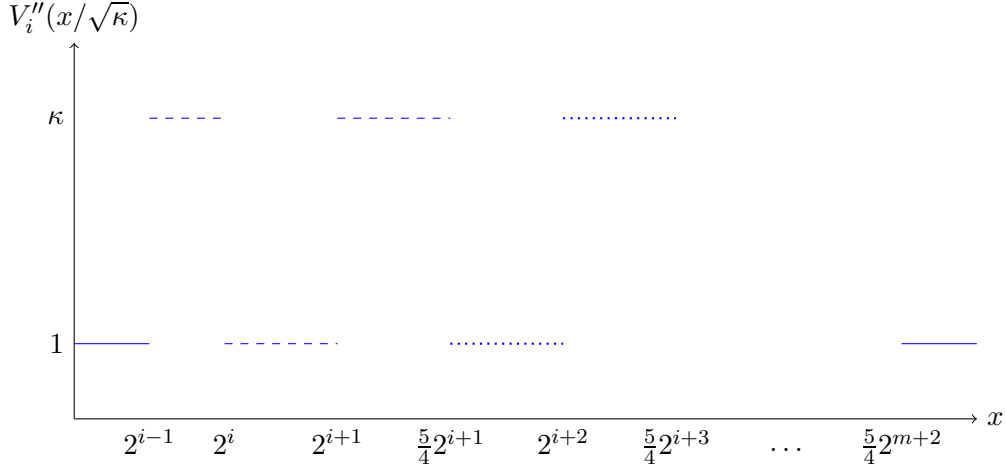

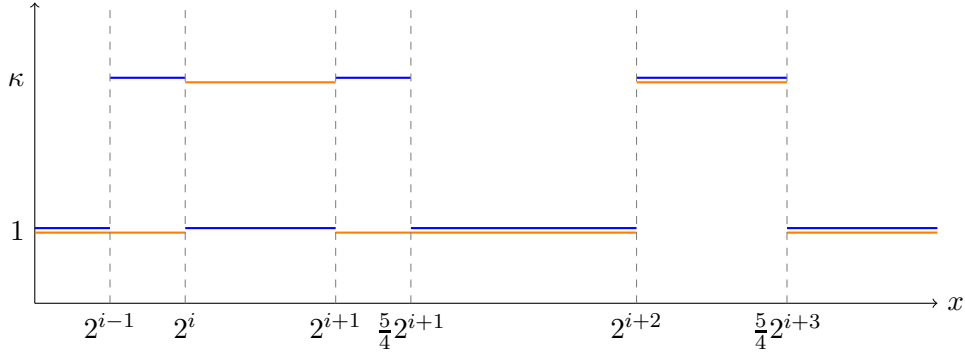
\begin{figure}
    \centering
    \begin{tikzpicture}
  \draw[->] (0, 0) -- (12, 0) node[right] {$x$};
  \draw[->] (0, 0) -- (0, 4) node[above] {};
  \draw[blue, thick] (0, 1) -- (1, 1);
  \draw[blue, thick] (1, 3) -- (2, 3);
  \draw[blue, thick] (2, 1) -- (4, 1);
  \draw[blue, thick] (4, 3) -- (5, 3);
  \draw[blue, thick] (5, 1) -- (8, 1);
  \draw[blue, thick] (8, 3) -- (10, 3);
  \draw[blue, thick] (10, 1) -- (12, 1);
  %\draw[blue, thick] (6.5, 3) -- (8, 3);
  %\draw[blue, thick] (11, 1) -- (12, 1);
  \node at (0, 0.97) [left] {$1$};
  \node at (0, 2.97) [left] {$\kappa$};
  \node at (1, 0) [below] {$2^{i-1}$};
  \node at (2, 0) [below] {$2^i$};
  \node at (4, 0) [below] {$2^{i+1}$};
  \node at (5, 0) [below] {$\frac{5}{4}2^{i+1}$};
  \node at (8, 0) [below] {$2^{i+2}$};
  \node at (10, 0) [below] {$\frac{5}{4}2^{i+3}$};
  %\node at (9.5, -0.25) [below] {$\dots$};
  %\node at (11, 0) [below] {$\frac{5}{4}t 2^{m+2}$};
  \draw[orange, thick] (0, 0.94) -- (2, 0.94);
  \draw[orange, thick] (2, 2.94) -- (4, 2.94);
  \draw[orange, thick] (4, 0.94) -- (8, 0.94);
  \draw[orange, thick] (8, 2.94) -- (10, 2.94);
  \draw[orange, thick] (10, 0.94) -- (12, 0.94);
  % V_{i+1}
%   \draw[orange] (0, 0.97) -- (2, 0.97);
%   \draw[orange, dashed] (2, 4) -- (3.5, 4);
%   \draw[orange, dashed] (3.5, 0.97) -- (6.5, 0.97);
%   \draw[orange, dashed] (6.5, 3.97) -- (8, 3.97);
    \draw[gray, dashed] (1, 0) -- (1, 4);
    \draw[gray, dashed] (2, 0) -- (2, 4);
    \draw[gray, dashed] (4, 0) -- (4, 4);
    \draw[gray, dashed] (5, 0) -- (5, 4);
    \draw[gray, dashed] (8, 0) -- (8, 4);
    \draw[gray, dashed] (10, 0) -- (10, 4);
\end{tikzpicture}
    \caption{We plot $V_i''$ (in blue) and $V_{i+1}''$ (in orange). In this figure, we do not distort the horizontal axis lengths to make it easier to visually compare the relative lengths of intervals on which the second derivatives are constant.}
    \label{fig:comparison}
\end{figure}

The following lemma provides intuition for the construction.

\begin{lemma}
    We have the equalities
    \begin{align*}
        V_i
        &= V_{i+1}\,, \\
        V_i'
        &= V_{i+1}'\,, \\
        V_i''
        &= V_{i+1}''\,,
    \end{align*}
    outside of the set $\{x\in\R : \ttt 2^{i-1} \le \abs x \le \frac{5}{4} \ttt 2^{i+1}\}$.
\end{lemma}
\begin{proof}
    Refer to Figure~\ref{fig:comparison} for a visual aid for the proof.

    Clearly the potentials and derivatives match when $\abs x \le \ttt 2^{i-1}$.
    Since the second derivatives match when $\abs x \ge \frac{5}{4} \ttt 2^{i+1}$, it suffices to show that $V_i'( \frac{5}{4}\ttt 2^{i+1}) = V_{i+1}'( \frac{5}{4}\ttt 2^{i+1})$ and $V_i( \frac{5}{4}\ttt 2^{i+1}) = V_{i+1}( \frac{5}{4}\ttt 2^{i+1})$.

    To that end, note that for $x\ge 0$,
    \begin{align*}
        V_{i+1}''(x) - V_i''(x)
        &= \one\{\ttt 2^{i-1} < x \le \ttt 2^i\} - \phi\bigl( \frac{x}{\ttt 2^i} \bigr) +\phi\bigl( \frac{x}{\ttt 2^{i+1}}\bigr) -\psi\bigl( \frac{x}{\ttt 2^i}\bigr) \\
        &= \begin{cases} -(\kappa-1)\,, & \ttt 2^{i-1} \le x \le \ttt 2^i\,, \\ +(\kappa-1)\,, & \ttt 2^i \le x \le \ttt 2^{i+1}\,, \\ -(\kappa-1)\,, & \ttt 2^{i+1} \le x \le \frac{5}{4} \ttt 2^{i+1}\,, \\ 0\,, & \text{otherwise}\,. \end{cases}
    \end{align*}
    A little algebra shows that the above expression integrates to zero, hence we deduce the equality $V_i'( \frac{5}{4} \ttt 2^{i+1}) = V_{i+1}'( \frac{5}{4} \ttt 2^{i+1})$.
    Also, by integrating this expression twice, we see that
    \begin{align*}
        V_{i+1}\bigl( \frac{5}{4} \ttt 2^{i+1}\bigr) - V_i\bigl( \frac{5}{4} \ttt 2^{i+1}\bigr)
        &= \underbrace{- \frac{\kappa-1}{2} \, {(\ttt 2^{i-1})}^2}_{\text{integral on}~[\ttt 2^{i-1}, \ttt 2^i]} \\
        &\qquad{} \underbrace{-(\kappa - 1) \ttt 2^{i-1} \, \ttt 2^i + \frac{\kappa-1}{2} \, {(\ttt 2^i)}^2}_{\text{integral on}~[\ttt 2^i, \ttt 2^{i+1}]} \\
        &\qquad{} \underbrace{+ (\kappa -1) \ttt 2^{i-1} \, \frac{1}{4} \ttt 2^{i+1} - \frac{\kappa-1}{2} \, \bigl( \frac{1}{4} \ttt 2^{i+1} \bigr)^2}_{\text{integral on}~[\ttt 2^{i+1}, \frac{5}{4} \ttt 2^{i+1}]} \\
        &= \frac{\kappa-1}{\kappa} \, \{-2^{2i-3} - 2^{2i-1} + 2^{2i-1} + 2^{2i-2} - 2^{2i-3}\} \\
        &= 0\,,
    \end{align*}
    as desired.
\end{proof}

We also need a lemma showing that each probability distribution $p_i \propto \exp(-V_i)$ places a substantial amount of mass on the interval $(\ttt 2^{i-2}, \ttt 2^{i-1}]$.

\begin{lemma}\label{lem:low_bd_lem2}
    For each $i \in [m]$,
    \begin{align*}
        p_i\bigl((\ttt 2^{i-2}, \ttt 2^{i-1}]\bigr)
        \ge \frac{1}{32}\,.
    \end{align*}
\end{lemma}
\begin{proof}
    According to the definition of $p_i$, we have
    \begin{align*}
        p_i\bigl((\ttt 2^{i-2}, \ttt 2^{i-1}]\bigr)
        &= \frac{\int_{\ttt 2^{i-2}}^{\ttt 2^{i-1}} \exp(-x^2/2) \, \D x}{Z_{p_i}}\,, \qquad Z_{p_i} := \int_\R \exp(-V_i)\,.
    \end{align*}
    Recalling that $m$ is chosen so that $\exp(-x^2/2) \ge 1/2$ whenever $\abs x \le \ttt 2^{m-1}$ (see~\eqref{eq:lower_bd_defn_m}), we can conclude that
    \begin{align*}
        \int_{\ttt 2^{i-2}}^{\ttt 2^{i-1}} \exp\bigl( - \frac{x^2}{2}\bigr) \, \D x
        &\ge \frac{1}{2} \, \ttt 2^{i-2}\,.
    \end{align*}
    For the normalizing constant, observe that
    \begin{align*}
        \int_0^\infty \exp(-V_i)
        &= \int_0^{\ttt 2^i} \exp(-V_i) + \int_{\ttt 2^i}^\infty \exp(-V_i)
        \le \ttt 2^i + \int_{\ttt 2^i}^\infty \exp(-V_i)\,.
    \end{align*}
    Since $V_i'' = \kappa$ on $[\ttt 2^{i-1}, \ttt 2^i]$, it follows that $V'(\ttt 2^i) \ge \kappa^{\frac{1}{2}} 2^{i-1}$, and so
    \begin{align*}
        V_i(x)
        &\ge \kappa^{\frac{1}{2}} 2^{i-1} \, (x-\ttt 2^i) + \frac{{(x - \ttt 2^i)}^2}{2}\,, \qquad x \ge \ttt 2^i\,.
    \end{align*}
    Therefore,
    \begin{align*}
        \int_{\ttt 2^i}^\infty \exp(-V_i)
        &\le \int_{\ttt 2^i}^\infty \exp\bigl(-\kappa^{\frac{1}{2}} 2^{i-1} \, (x-\ttt 2^i) - \frac{{(x - \ttt 2^i)}^2}{2}\bigr) \, \D x
        \le \frac{1}{\kappa^{\frac{1}{2}} 2^{i-1}}
        \le \frac{1}{\sqrt\kappa}\,,
    \end{align*}
    where we applied a standard tail estimate for Gaussian densities (\autoref{lem:gaussian_tails}).
    Putting it together,
    \begin{align*}
        p_i\bigl((\ttt 2^{i-2}, \ttt 2^{i-1}]\bigr)
        &\ge \frac{2^{i-3}}{2\,(2^i + 1)}
        \ge \frac{1}{32}\,,
    \end{align*}
    which proves the result.
\end{proof}

\subsection{Lower bound via Fano's inequality}

In this section, we use the densities $\{p_i\}_{i=1}^m$ constructed in the previous section together with Fano's inequality from information theory in order to prove the lower bound.

\begin{proof}[Proof of~\autoref{thm:general_lower}]
    Let $Z \sim \unif([m])$ be an index chosen uniformly at random.
    Suppose that an algorithm makes $n$ queries to the oracle for $p_Z$, and given $Z = i$, outputs a sample $Y$ whose law $q_i$ is at total variation distance at most $\frac{1}{64}$ from $p_i$.
    %\patrik{We might want to be more precise, I think the guarantee is that $TV(q_i := \mathcal{L}(Y|Z=i), p_i) \leq 1/100$ (?)}. 
  In light of Lemma~\ref{lem:low_bd_lem2}, a good candidate estimator for $Z$ from the observation of $Y$  is given by
    \begin{align*}
        \widehat Z
        &:=\{k \in \N : Y \in (\ttt 2^{k-2}, \ttt 2^{k-1}]\}\,.
    \end{align*}
    On the one hand, the probability that the estimator is correct is bounded by
    \begin{align}
        \Pr\{\widehat Z = Z\}
        &= \frac{1}{m} \sum_{i=1}^m \Pr\{\widehat Z = i \mid Z = i\}
        = \frac{1}{m} \sum_{i=1}^m \Pr\{Y \in (\ttt 2^{i-2}, \ttt 2^{i-1}] \mid Z = i\} \nonumber \\
        &= \frac{1}{m} \sum_{i=1}^m q_i\bigl((\ttt 2^{i-2}, \ttt 2^{i-1}]\bigr)
        \ge \frac{1}{m} \sum_{i=1}^m p_i\bigl((\ttt 2^{i-2}, \ttt 2^{i-1}]\bigr) - \frac{1}{64}
        \ge \frac{1}{64}\,, \label{eq:prob_correct_guess}
    \end{align}
    where the last inequality uses~\autoref{lem:low_bd_lem2}.
    
    On the other hand, we can lower bound $\Pr\{\widehat Z \ne Z\}$ using Fano's inequality. Let $x_1,\dotsc,x_n$ denote the query points of the algorithm, and let $W_i$ be a shorthand for the triple $(V_i, V_i', V_i'')$. We will first prove the lower bound for deterministic algorithms, i.e., assuming that each query point $x_j$ is a deterministic function of the previous query points and query values.
    % As a shorthand, write $W_i:=(V_i, V_i', V_i'')$. Let $\mathcal{A}$ be any algorithm that, based on $n$ queries $x_1,\dotsc,x_n$ to the oracle, constructs an estimator of $Z$ denoted $\mathcal{A}(n)$. Let $\xi$ be the random seed of the algorithm $\mathcal{A}$, so that its queries may be written as
% \begin{align*}
%     x_1 &= f_1(\xi) \\
%     x_2 &= f_2(\xi, (x_1, W_Z(x_1))) \\
%     &\vdots \\
%     x_n &= f_n(\xi, (x_i, W_Z(x_i))_{i < n}) \\
%     \mathcal{A}(n) &= f_{n+1}(\xi, (x_i, W_Z(x_i))_{i \in [n]})
% \end{align*}
%     for deterministic, measurable functions $f_1, \dots, f_{n+1}$.
    Since
    \begin{equation*}
        Z \to \{x_j, W_Z(x_j), \; j \in [n]\} \to \widehat Z
    \end{equation*}
    forms a Markov chain, Fano's inequality~\citep{coverthomas2006infotheory} yields
    \begin{align*}
        \Pr\{\widehat Z \ne Z\}
        &\ge 1 - \frac{I({\{x_j, W_Z(x_j)\}}_{j\in [n]} ; Z) + \log 2}{\log m}\,,
    \end{align*}
    where $I$ denotes the mutual information. By the chain rule for mutual information \citep{coverthomas2006infotheory},
    \begin{equation*}
        I\bigl({\{x_j, W_Z(x_j)\}}_{j\in [n]} ; Z\bigr)
        = \sum_{j=1}^n I\bigl(x_j, W_Z(x_j); Z \bigm\vert x_1, W_Z(x_1),\dotsc,x_{j-1}, W_Z(x_{j-1})\bigr)\,.
    \end{equation*}
    Observe that, conditioned on ${\{x_i, W_Z(x_i)\}}_{i=1}^{j-1}$, the query point $x_j$ is deterministic. Also, from the construction of the family of potentials, we know that $W_Z(x_j) = W_1(x_j)$ if $x_j \le \ttt 2^{Z-1}$, and $W_Z(x_j) = W_m(x_j)$ if $x_j \ge \frac{5}{4} \ttt 2^{Z+1}$. It yields that:
    \begin{itemize}
        \item for $Z \le \log_2(\frac{4}{5}\sqrt \kappa x_j)-1$, $W_Z(x_j)$ takes a unique value given by $W_m(x_j)$,
        \item for $Z \ge \log_2(\sqrt \kappa x_j)+1$, $W_Z(x_j)$ takes a unique value given by $W_1(x_j)$,
    \end{itemize} 
    and otherwise, $Z$ lives in an interval of size at most $\log_2(\sqrt \kappa x_j)+1-(\log_2(\frac{4}{5}\sqrt \kappa x_j)-1) \le 2+\log_2(5/4)$ which covers at most three integers, say $z_0-1, z_0, z_0+1$. 
    Hence, the conditional distribution of $W_Z(x_j)$ can be supported on at most $5$ points given respectively by \[ W_1(x_j), W_m(x_j), W_{z_0-1}(x_j), W_{z_0}(x_j),~\text{and}~ W_{z_0+1}(x_j)\,. \] Since the mutual information is upper bounded by the conditional entropy of $W_Z(x_j)$, we can conclude
    \begin{align*}
        I\bigl({\{x_j, W_Z(x_j)\}}_{j\in [n]} ; Z\bigr)
        &\le n\log 5\,.
    \end{align*}
    Substituting this into Fano's inequality yields
    \begin{align}\label{eq:low_bd_deterministic}
        \Pr\{\widehat Z \ne Z\}
        &\ge 1 - \frac{n\log 5 + \log 2}{\log m}\,.
    \end{align}
    
    In general, if the algorithm is randomized, then we can apply the inequality~\eqref{eq:low_bd_deterministic} conditioned on the random seed $\xi$ of the algorithm, since $\xi$ is independent of $Z$. It yields
    \begin{align*}
        \Pr\{\widehat Z \ne Z \mid \xi\}
        &\ge 1 - \frac{n\log 5 + \log 2}{\log m}\,,
    \end{align*}
    and upon taking expectations we see that~\eqref{eq:low_bd_deterministic} holds for randomized algorithms as well.
    
    % Putting it all together, we have
    % \begin{align*}
    %     \Pr(\widehat Z \ne Z) &\geq \min\limits_{\mathcal{A}} \Pr(\mathcal{A}(n) \ne Z) = \min\limits_{\mathcal{A}} \E \Pr(\mathcal{A}(n) \neq Z \,|\,\xi) \\
    %     &\geq 1 - \frac{n \log 5 + \log 2}{\log m}.
    % \end{align*}
    Combined with~\eqref{eq:prob_correct_guess}, we obtain $n \gtrsim \log m \gtrsim \log\log\kappa$ as desired. 
\end{proof}

\section{Proof of the upper bound}\label{scn:stronglylc_upper}

Let $p$ be the target distribution and let $\tilde p = p Z_p$ denote the unnormalized distribution which we access via oracle queries. We recall our preprocessing steps: we assume that the query values take the form $\tilde p(x) = \exp(-V(x))$, with $V(0) = V'(0) = 0$ and $V$ satisfying~\eqref{eq:smooth_strong_cvx}. This is without loss of generality because we can query $\tilde p(0)$ and replace subsequent queries $\tilde p(x)$ with $\tilde p(x)/\tilde p(0)$, thereby normalizing $V$ to satisfy $V(0) = 0$. By rescaling the distribution, we can assume that $1 \le V'' \le \kappa$. Also, we can assume that the target distribution is only supported on the positive reals $\R_+$, because we can then construct an upper envelope on all of $\R$ by repeating our algorithm on the negative reals, which only doubles the number of queries and does not change the complexity.

\begin{proof}[Proof of Theorem~\ref{thm:upper_bd}]
Our goal is to use the oracle queries to construct an upper envelope $\tilde q$ that satisfies $\tilde q \ge \tilde p$, and $Z_q \lesssim Z_p$, where
\begin{align*}
    Z_p
    &:= \int_\R \tilde p\,, \qquad Z_q := \int_\R \tilde q
\end{align*}
are the normalizing constants. The guarantees of Theorem~\ref{thm:upper_bd} will then follow from standard results on rejection sampling.
For completeness, we state and prove the relevant result as~\autoref{thm:rejection sampling} in Appendix~\ref{scn:auxiliary}.

Let $i_0$ denote the smallest integer such that  $V(2^{i_0}/\sqrt \kappa) \ge 1/2$. Note that $x^2/2 \le V(x) \le \kappa x^2/2$ implies that $0 \le i_0 \le  (\log_2 \kappa)/2$.
% $i_0 \in \{0,1,\dotsc,\lceil \frac{1}{2} \log_2 \kappa\rceil\}$ such that $V(2^{i_0}/\sqrt \kappa) \ge 1/2$.
% Such an index must exist; indeed, since $x^2/2 \le V(x) \le \kappa x^2/2$, the former inequality shows that $i_0 \le \lceil \frac{1}{2} \log_2 \kappa\rceil$ and the latter inequality shows that $i_0 \ge 0$.
Using binary search over an array of size $\mc O(\log \kappa)$, we can find $i_0$ using only $\mc O(\log \log \kappa)$ queries to $\tilde p$.

Let $x_0 := 2^{i_0}/\sqrt \kappa$.
We first claim that
\begin{equation}
\label{eq:pr:norm_const_tildep}
        \int_0^{x_0} \tilde p \gtrsim \, x_0\,.
\end{equation}
When $i_0 = 0$, this holds because
\begin{align*}
    \int_0^{x_0} \tilde p
    &= \int_0^{1/\sqrt \kappa} \exp(-V)
    \ge \int_0^{1/\sqrt \kappa} \exp\bigl( - \frac{\kappa x^2}{2} \bigr) \, \D x\ge \frac{1}{3\sqrt\kappa}=\frac{x_0}{3}\,.
    % \gtrsim\, x_0\,.
\end{align*}
When $i_0 > 0$, this holds because, by definition of $i_0$, we have $V(x_0/2) \le 1/2$, and so
\begin{align*}
    \int_0^{x_0} \tilde p
    &\ge \int_0^{x_0/2} \exp(-V)
    \gtrsim \, x_0\,.
\end{align*}

\medskip

Next,  define the upper envelope as follows:
\begin{align*}
    \tilde q(x) = \begin{cases} 1\,, & x \le x_0\,, \\ \exp\{- (x-x_0)/(2x_0) - {(x-x_0)}^2/2\}\,, & x > x_0\,. \end{cases}
\end{align*}
To see that $\tilde q \ge \tilde p$ and hence that $\tilde q$ is a valid upper envelope, observe first that  since $\tilde p(0) = 1$, and $\tilde p$ is decreasing, we get that $\tilde p(x)\le 1=\tilde q(x)$ for all $x \in [0, x_0]$. 

Next, if $x > x_0$, using the fact that $V$ is convex and $V(x_0)\ge 1/2$ by the definition of $x_0$,
\begin{align*}
    V'(x_0)
    &\ge \frac{V(x_0) - V(0)}{x_0}
    \ge \frac{1}{2x_0}\,.
\end{align*}
Hence, for any $x > x_0$ we have
\begin{align*}
    V(x)
    &\ge V(x_0) + V'(x_0) \, (x-x_0) + \frac{1}{2} \, {(x-x_0)}^2 \\
    &\ge \frac{1}{2x_0} \, (x-x_0) + \frac{1}{2} \, {(x-x_0)}^2\,.
\end{align*}
It implies that $\tilde p(x)\le \tilde q(x)$ also for the tail $x>x_0$.

To complete the proof, we  show that $Z_q \lesssim Z_p$. In light of~\eqref{eq:pr:norm_const_tildep} it is sufficient to show that $Z_q \lesssim x_0$. To see this, observe that by~\autoref{lem:gaussian_tails}, we have
\begin{align*}
 Z_q=  \int_0^{x_0} \tilde q+ \int_{x_0}^\infty \tilde q \le x_0+\int_{x_0}^\infty \exp\bigl( - \frac{1}{2x_0} \, (x-x_0) + \frac{1}{2} \, {(x-x_0)}^2\bigr) \, \D x
    \le 3x_0\,.
\end{align*}
This completes the proof.
\end{proof}

\section{Proof for the hit-and-run algorithm}\label{sec:proof of hit and run}

In this section, we prove Proposition~\ref{prop:hit-and-run}.

Let $p \propto \exp(-V)$ where $V:\R^d \to \R$ is $1$-strongly convex and $\kappa$-smooth with $V(0) = \nabla V(0) = 0$. We are given a line $\ell = \{x^\star + \lambda u \,:\, \lambda \in \R\}$ where $x^\star$ is the closest point of $\ell$ to the origin and $u$ is a unit vector. Note that
\begin{equation*}
    1 \leq u^\T \nabla^2 V(x^\star + \lambda u) u  = \frac{\D^2 V(x^\star + \lambda u)}{\D \lambda^2} \leq \kappa
\end{equation*}
so that $V$ restricted to $\ell$ is also $1$-strongly convex and $\kappa$-smooth. 

First, we need to find the minimum of $V$ on $\ell$. To that end, we restrict the line $\ell$ to the subset of points $x^\star+ \lambda u$ with 
\begin{equation}\label{eq:bd_on_lambda}
    |\lambda| \leq  2 \kappa \,|x^\star|_2\,,
\end{equation}
where  we write $\abs\cdot_2$ for the Euclidean norm on $\R^d$.

To prove~\eqref{eq:bd_on_lambda}, note that by strong convexity of $V$ on $\ell$, if $x^\star+\lambda u$ is a minimizer of $V$ on $\ell$, then we must have
\begin{equation*}
    \lambda \nabla V(x^\star)^\T u + \frac{1}{2}\, \lambda^2 \leq V(x^\star + \lambda u) - V(x^\star) \le 0\,.
\end{equation*}
% first that the right-hand side of~\eqref{eq:bd_on_lambda} follows from the $\kappa$-smoothness of $V$. To establish the 

% Note that
% \begin{equation*}
%     1 \leq u^\T \nabla^2 V(x^\star + \lambda u) u  = u^\T \frac{\D^2 V(x^\star + \lambda u)}{\D \lambda^2} u \leq \kappa
% \end{equation*}
% so that $V$ restricted to $\ell$ is also $1$-strongly convex and $\kappa$-smooth. First, we need to find the minimum of $V$ on $\ell$. For 
% for any $\lambda \in \R$
% \begin{equation*}
%     \lambda \nabla V(x^\star)^\T u + \frac{1}{2}\, \lambda^2 \leq V(x^\star + \lambda u) - V(x^\star)\,.
% \end{equation*}
% For $x^\star+\lambda u$ to be  minimizer, we need $V(x^\star + \lambda u) < V(x^\star)$ we certainly need the LHS of the above display to be smaller than $0$, which implies the left-hand side of the following inequality
The above yields $ |\lambda| \leq 2\,|\nabla V(x^\star)^\T u|$. Moreover, by $\kappa$-smoothness of $V$ we have 
$$
|\nabla V(x^\star)^\T u| \le |\nabla V(x^\star)|_2 \le \kappa |x^\star|_2\,,
$$
which completes the proof of~\eqref{eq:bd_on_lambda}.
% \begin{equation}\label{eq:bd_on_lambda}
%     |\lambda| \leq 2\,|\nabla V(x^\star)^\T u| \,.
% \end{equation}
% where the second in
% The right-hand side of~\eqref{eq:bd_on_lambda} simply follows from the $\kappa$-smoothness of $V$. 
% Here, we write $\abs\cdot_2$ for the Euclidean norm on $\R^d$.

%We now restrict our attention to $\lambda$ satisfying~\eqref{eq:bd_on_lambda}. 
To ease notation, from here on we write $V$ for its one-dimensional restriction to $\ell$. We modify our rejection sampling algorithm from~\autoref{thm:upper_bd} so that it works when (i) the minimizer $m \in \R$ of $V$ is only known approximately and (ii) the value $V(m)$ is known only approximately. To make this precise, suppose that we know $a,b \in \R$ such that $a \leq m \leq b$. By~\eqref{eq:bd_on_lambda}, such $a,b$ can be found using $\mathcal{O}(\log(\kappa \, |x^\star|_2 / (b-a)))$ queries via bisection. Strong convexity and smoothness give us
\begin{align}
    V(m) + \frac 12\, (b-m)^2 &\leq V(b) \leq V(m) + \frac{\kappa}{2}\,(b-m)^2\,,\label{eqn:approx mode bound1} \\
    V(m) + \frac 12 \,(m-a)^2 &\leq V(a) \leq V(m) + \frac{\kappa}{2}\,(m-a)^2\,.\label{eqn:approx mode bound2}
\end{align}
In particular, 
\begin{equation*}
    V(a) \lor V(b) - \frac{\kappa}{2} \,(b-a)^2 \leq V(m) \leq V(a) \land V(b)\,.
\end{equation*}
Relabel $V - (V(a)\lor V(b))$ as $V$, so that we may assume 
\begin{equation*}
    -\frac{\kappa}{2}\,(b-a)^2 \leq V(m) \leq 0\,.
\end{equation*}
Assume from here on that $\kappa\,(b-a)^2/2 = 1$ (this means we used $\mathcal{O}(\log(\kappa \,|x^\star|_2))$ queries). Define
\begin{align*}
    i_b &:= \min\bigl\{i\in\N\,:\,V\bigl(b+\frac{2^i}{\sqrt{\kappa}}\bigr) \geq 3\bigr\}\,, \\
    i_a &:= \min\bigl\{i\in\N\,:\,V\bigl(a - \frac{2^i}{\sqrt{\kappa}}\bigr) \geq 3 \bigr\}\,.
\end{align*}
Let $x_b = b+2^{i_b}/\sqrt{\kappa}$ and $x_a = a-2^{i_a}/\sqrt{\kappa}$. Note that $1\le i_a, i_b \lesssim \log \kappa$. Indeed, for the lower bound observe that by \eqref{eqn:approx mode bound1}, we have
\begin{equation*}
    V\bigl(b+\frac{1}{\sqrt{\kappa}}\bigr) \leq V(m) + \frac{\kappa}{2} \,\bigl(b+\frac{1}{\sqrt{\kappa}} - a\bigr)^2 \leq 0 + 1 + \frac 12 + \sqrt{2} < 3\,, 
\end{equation*}
so that $i_b\ge 1$.  

For the upper bound, we have
\begin{equation*}
    V\bigl(b + \frac{2^{\frac 12 (\log_2 \kappa + 3)}}{\sqrt{\kappa}}\bigr) \geq V(m) + \frac 12 \, \bigl(b+\frac{2^{\frac 12 (\log_2 \kappa + 3)}}{\sqrt{\kappa}}-m\bigr)^2 \geq -1 + \frac{2^{\log_2 \kappa + 2}}{\kappa} = 3\,.
\end{equation*}
It yields $i_b \lesssim \log \kappa$. Similarly \eqref{eqn:approx mode bound2} yields the desired bounds for $i_a$.

Thus, $i_a, i_b$ may be found via binary search in $\mathcal{O}(\log\log\kappa)$ queries to $\tilde p$. 

We now move to the definition of the upper envelope $\tilde q$. To that end, note first that 
\begin{equation*}
    V'(x_b) \geq \frac{V(x_b) - V(m)}{x_b-m} \geq \frac{3}{x_b - a}
\end{equation*}
so that for $x \geq x_b$ we can write
\begin{align*}
    V(x) &\geq V(x_b) + V'(x_b)(x-x_b) + \frac 12\, (x-x_b)^2 \\
    &\geq 3 + 3\, \frac{x-x_b}{x_b-a} + \frac 12\, (x-x_b)^2\,.
\end{align*}
Similarly, we also have
\begin{equation*}
    V(x) \geq 3 + 3\, \frac{x - x_a}{x_a - b} + \frac 12\, (x-x_a)^2
\end{equation*}
for $x \leq x_a$. We are ready for the definition of $\tilde{q}$. 
\begin{equation*}
    \tilde{q}(x) ;= \begin{cases}
    e\,, &\text{for } x \in [x_a, x_b]\,, \\
    \exp(- 3 - 3\, \frac{x-x_b}{x_b - a} - \frac 12\, (x-x_b)^2)\,, &\text{for } x \geq x_b\,, \\
    \exp(- 3 - 3 \,\frac{x-x_a}{x_a - b} - \frac 12 \,(x-x_a)^2)\,, &\text{for } x \leq x_a\,.
    \end{cases}
\end{equation*}
By our above calculations we know that $\tilde{p} \leq \tilde{q}$ on all of $\R$. We further have
\begin{equation*}
   Z_p\ge  \int_{x_a}^{x_b} \tilde{p} \geq \int_{a - 2^{i_a-1}/\sqrt{\kappa}}^{b + 2^{i_b-1}/\sqrt{\kappa}} \exp(-V) \geq \frac{\exp(-3)}{2}\,(x_b-x_a).
\end{equation*}
Thus, to conclude, it suffices to show that $Z_q = \int_\R \tilde{q} \lesssim x_b-x_a$. This follows by~\autoref{lem:gaussian_tails}.

Given that the chain is at point $x_t$ at some time $t$, we have described an algorithm that constructs an efficient rejection envelope with number of queries bounded by
\begin{equation*}
    \mc O\bigl(\log(\kappa\, |x_t|_2) \vee 1 + \log\log\kappa\bigr)\,. 
\end{equation*}
To bound the amortized query complexity, it remains to control the time average of $|x_t|_2$. To that end, recall that the Hit-and-Run chain is geometrically ergodic by \cite{lovasz2003hit} so  that, 
\begin{equation*}
    \frac 1T \sum\limits_{t=1}^T \log(\kappa \,|x_t|_2) \lor 1 \stackrel{\text{a.s.}}{\longrightarrow} \E_{x \sim p} \log(\kappa\,|x|_2) \lor 1.
\end{equation*}
 By Jensen's inequality we further have
\begin{align*}
    \E_{x \sim p} \log(\kappa\,|x|_2) \lor 1 &\lesssim \log(\kappa \E_{x \sim p} |x|_2)\lor 1 \lesssim \log(\kappa d)\,,
\end{align*}
since $p$ is $1$-sub-Gaussian. Noting that $\log\log \kappa \leq \log(\kappa d)$, the result follows.

\section{Auxiliary results}\label{scn:auxiliary}

The following result on rejection sampling is standard, and we include it for the sake of completeness.

\begin{theorem}\label{thm:rejection sampling}
    Suppose we have query access to the unnormalized target $\tilde{p} = p Z_p$ supported on $\ms X$, and that we have an upper envelope $\tilde q \ge \tilde p$.
    Let $q$ denote the corresponding normalized probability distribution and write $Z_q$ for the normalizing constant, i.e., $\tilde q = q Z_q$.
    % \begin{enumerate}
    %     \item We can sample from $q$ efficiently.
    %     \item $p(x)/q(x) \leq M$ for all $x \in \ms X$, where $M$ is finite and known. 
    %     \item $Z_p/Z_q \leq C$ where $C$ is finite and known.
    % \end{enumerate}
    Then, rejection sampling with acceptance probability $\tilde{p}/\tilde q$ outputs a point distributed according to $p$, and the number of samples drawn from $q$ until a sample is accepted follows a geometric distribution with mean $Z_q/Z_p$.
\end{theorem}
\begin{proof}
Since $\tilde q$ is an upper envelope for $\tilde p$, then $\tilde p(X)/\tilde q(X) \le 1$ is a valid acceptance probability.
Clearly, the number of rejections follows a geometric distribution. The probability of accepting a sample is given by
\begin{align*}
    \Pr(\text{accept})
    &= \int_{\ms X} \frac{\tilde p(x)}{\tilde q(x)} \, q(\D x)
    = \frac{Z_p}{Z_q} \int_{\ms X} p(\D x)
    = \frac{Z_p}{Z_q}\,.
\end{align*}
Let $X_1, X_2, X_3\dots$ be a sequence of i.i.d. samples from $q$ and let $U_1, U_2, U_3\dots$ be i.i.d. $\unif[0,1]$. Let $A \subseteq \ms X$ be a measurable set, and let $X$ be the output of the rejection sampling algorithm. Partitioning by the number of rejections, we may write
\begin{align*}
    \Pr(X \in A) &= \sum\limits_{n=0}^\infty \Pr\Bigl(X_{n+1} \in A,\; U_i > \frac{\tilde{p}(X_i)}{\tilde{q}(X_i)}\, \forall\,i\in [n],\; U_{n+1} \leq \frac{\tilde{p}(X_{n+1})}{\tilde{q}(X_{n+1})}\Bigr) \\
    &= \sum\limits_{n=0}^\infty \Pr\Bigl(X_{n+1} \in A,\, U_{n+1} \leq \frac{\tilde{p}(X_{n+1})}{\tilde{q}(X_{n+1})}\Bigr) \, \Pr\Bigl(U_1 > \frac{\tilde{p}(X_1)}{\tilde{q}(X_1)}\Bigr)^n \\
    &= \sum\limits_{n=0}^\infty \Bigl(\int_A \frac{\tilde{p}(x)}{\tilde{q}(x)} \, q(\D x)\Bigr) \Bigl(\int_{\ms X} \bigl(1 - \frac{\tilde{p}(x)}{ \tilde{q}(x)}\bigr) \, q(\D x)\Bigr)^n \\
    &= p(A) \, \frac{Z_p}{Z_q} \sum\limits_{n=0}^\infty \bigl(1 - \frac{Z_p}{Z_q}\bigr)^n
    = p(A)\,.
\end{align*}
\end{proof}

We also use the following elementary lemma about Gaussian integrals.

\begin{lemma}\label{lem:gaussian_tails}
    Let $a, x_0 > 0$.
    Then,
    \begin{align*}
        \int_{x_0}^\infty \exp\bigl( - a \, (x-x_0) - \frac{1}{2} \, {(x-x_0)}^2\bigr) \, \D x
        \le \frac{1}{a}\,.
    \end{align*}
\end{lemma}
\begin{proof}
    Completing the square,
    \begin{align*}
        \int_{x_0}^\infty \exp\bigl( - a \, (x-x_0) - \frac{1}{2} \, {(x-x_0)}^2\bigr) \, \D x
        &= \int_0^\infty \exp\bigl( - ax - \frac{1}{2} \, x^2\bigr) \, \D x \\
        & = \sqrt{2\pi} \exp\bigl( \frac{a^2}{2} \bigr) \Pr(Z>a)\,,
    \end{align*}
    where $Z \sim \mc N(0,1)$. 
    The result follows from the Mills ratio inequality~\citep{Gor41}.
\end{proof}

\end{document}